\documentclass[11pt]{amsart}
\usepackage[utf8x]{inputenc}
\usepackage[english]{babel}
\usepackage[T1]{fontenc}
 
\usepackage{graphicx}
\usepackage{caption}
\usepackage{subcaption}

\usepackage{amssymb,amsmath}
\usepackage[hidelinks]{hyperref}
\usepackage{indentfirst}
\usepackage{enumerate,amsmath,amssymb, mathrsfs,mathtools}
\usepackage{appendix}
\usepackage{latexsym}
\usepackage{url}
\usepackage{color}
\usepackage{accents}
\usepackage{setspace}
\usepackage{pdfpages}
\usepackage{stmaryrd}
\usepackage{amsrefs}

\usepackage[margin=2.5cm]{geometry}
\allowdisplaybreaks

\makeatletter
\def\PrintDOI#1{%
	\href{#1}{\expandafter\@stripdoi#1\@nil}%
}
\def\@stripdoi#1://#2/#3\@nil{#3}
\makeatother

\BibSpec{article}{%
	+{}{\PrintAuthors} {author}
	+{,}{ } {title}
	+{, }{\textit } {journal}
	+{}{ \parenthesize} {date}
		+{,  }{no. } {volume}
	+{,}{ } {pages}
	+{,}{ } {note}
	+{, }{\PrintDOI} {url}
}

\ExplSyntaxOn

\ExplSyntaxOff
\BibSpec{book}{%
	+{}{\PrintAuthors}  {author}
	+{. }{}{title}
	+{,}{ }{series}
	+{,}{ vol.~}{volume}
	+{. }{\textit}{publisher}
	+{,}{ ISBN}{isbn} 
}
\BibSpec{collection.article}{%
	+{}{\PrintAuthors}{author}
	+{, }{}{title}
	+{, }{\textit}{booktitle}
	+{, }{ \DashPages}{pages}
	+{,}{ }{series}
	+{, }{}{volume}
	+{, }{\textit}{publisher}
	+{,}{ }{date}
		+{, }{\PrintDOI} {url}
}

\mathcode`l="8000
\begingroup
\makeatletter
\lccode`\~=`\l
\DeclareMathSymbol{\lsb@l}{\mathalpha}{letters}{`l}
\lowercase{\gdef~{\ifnum\the\mathgroup=\m@ne \ell \else \lsb@l \fi}}%
\endgroup

\def\XXint#1#2#3{{\setbox0=\hbox{$#1{#2#3}{\int}$ }
		\vcenter{\hbox{$#2#3$ }}\kern-.6\wd0}}

\newtheorem{prop}{Proposition}
\newtheorem{thm}[prop]{Theorem}
\newtheorem{lem}[prop]{Lemma}
\newtheorem{coro}[prop]{Corollary}

\newtheorem{rema}[prop]{Remark}

\title{Weakly stable solutions of Serrin's  problem}
\author{Michael Eichmair}
\address{
	\textnormal{Michael Eichmair \newline  \indent
		University of Vienna \newline \indent
		Faculty of Mathematics  \newline \indent
		Oskar-Morgenstern-Platz 1 \newline \indent
		1090 Vienna, 	Austria  \newline\indent 
		\href{https://orcid.org/0000-0001-7993-9536}{https://orcid.org/0000-0001-7993-9536} \newline\indent	
		\href{mailto:michael.eichmair@univie.ac.at}{michael.eichmair@univie.ac.at}}
}

\author{Thomas Koerber}
\address{\textnormal{Thomas Koerber  \newline \indent
		University of Vienna \newline \indent
		Faculty of Mathematics  \newline \indent
		Oskar-Morgenstern-Platz 1 \newline \indent 1090 Vienna,	Austria \newline\indent 
		\href{https://orcid.org/0000-0003-1676-0824}{https://orcid.org/0000-0003-1676-0824} \newline \indent
		\href{mailto:thomas.koerber@univie.ac.at}{thomas.koerber@univie.ac.at}}
}

\begin{document}
	\maketitle
	\begin{abstract}
We prove that weakly stable solutions of Serrin's problem are compact and therefore  round balls.
	\end{abstract}
\section{Introduction}
Let $n\geq 2$ and $ \Omega \subset \mathbb{R}^n$  a domain, i.e., open, nonempty, and such that $\bar \Omega$ is a properly embedded top-dimensional submanifold.   We denote the mean curvature of $\partial \Omega$ with respect to the outward pointing normal $\mu(\Omega)$ by $H$. We call $\Omega$ a solution of Serrin's problem \eqref{Serrin} if there is $u \in C^\infty (\bar \Omega)$ such that
\begin{align} \label{Serrin} 
\begin{dcases}
\quad	-\Delta u=n\qquad& \text{in $\Omega$},\\
\quad	u>0&\text{in $\Omega$},\\ 
\quad	u=0 &\text{on $\partial \Omega$, and}\\
\quad	|Du|=1&\text{on $\partial \Omega$}.
\end{dcases}
\end{align}
 Note that $\mu(\Omega)=-Du$ on $\partial \Omega$ in this setting.

Using the method of moving planes, J.~Serrin \cite{Serrin} has shown that if $\Omega$ is compact and a solution of \eqref{Serrin}, then $\Omega$ is a  ball of radius 1. Subsequently, H.~Weinberger \cite{Weinberger} has found a short alternative proof of this fact. Such a result fails if $\Omega$ is allowed to be noncompact. Indeed, for all integers $m,\,n$ with $1\leq m\leq n-1$, the cylinder
\begin{align} \label{cylinder} 
\{y\in \mathbb{R}^m:|y|<m/n\}\times \mathbb{R}^{n-m}
\end{align} 
is a solution of \eqref{Serrin}. In view of this observation, H.~Berestycki, L.~Caffarelli, and L.~Nirenberg \cite[p.~1110]{BCN} have conjectured that if $\Omega$ is noncompact and a solution of \eqref{Serrin}, then $\Omega$ is either a  ball of radius $1$ or congruent to one of the  cylinders described in \eqref{cylinder}. Indeed,  they have made a similar conjecture for solutions $\Omega$ of the semilinear over-determined problem
\begin{align} \label{Serrin generalized} 
	\begin{dcases}
		\quad	-\Delta u=f(u)\qquad& \text{in $\Omega$},\\
		\quad	u>0&\text{in $\Omega$},\\ 
		\quad	u=0 &\text{on $\partial \Omega$, and}\\
		\quad	|Du|=1&\text{on $\partial \Omega$},
	\end{dcases}
\end{align} 
where $f\in C^\infty(\mathbb{R})$. We note that, by the work of J.~Serrin \cite{Serrin}, if $\Omega$ is compact and a solution of \eqref{Serrin generalized}, then $\Omega$ is a ball.  Counterexamples to this conjecture have been found in, e.g., \cite{Sicbaldi,DelPino} for certain choices of $f$ and, notably, by 
  M.~Fall, I.~Minlend, and T.~Weth \cite{Fall} for the choice $f=n$; see also \cite{WangWei} for a related positive result in this direction. In this paper, we show that, in the case where $f=n$,  the conjecture of H.~Berestycki, L.~Caffarelli, and L.~Nirenberg holds under a natural additional assumption on $\Omega$. 

To explain this, recall that, by the pioneering work of H.~Alt and L.~Caffarelli \cite{altcafarelli}, solutions of \eqref{Serrin generalized} are given by $\Omega=\{y\in \mathbb{R}^n:u(y)>0\}$ where $u\in C^\infty(\mathbb{R}^n)$ is a critical point  of the functional
$$
J(u)=\frac12\,\int_{\{y\in \mathbb{R}^n:u(y)>0\}} |Du|^2-\int_{\{y\in \mathbb{R}^n:u(y)>0\}}F(u).
$$
Here, $F\in C^\infty(\mathbb{R})$ is such that $F'=f$;
 see also, e.g., \cite{WangWei}. In the case where $f=n$, $J$ is not bounded from below. It is therefore natural to study $u\in C^\infty(\mathbb{R}^n)$ that are critical  for $J$ given the relative volume of $\{y\in \mathbb{R}^n:u(y)>0\}$. 
In this context,  we call $\Omega$ weakly stable if 
\begin{align}\label{stable}
\int_{\partial \Omega} (f(0)-H)\,\phi^2\leq \int_{\Omega} (|D\phi|^2-f'(u)\,\phi^2)
\end{align} 
for all $\phi\in C_c^\infty(\bar \Omega)$  such that 
$$
\int_{\partial \Omega} \phi =0.
$$
We call $\Omega$  stable if  \eqref{stable} holds
for all $\phi \in C_c^\infty(\bar \Omega)$. As we explain in Appendix \ref{appendix a}, $\Omega$ is weakly stable if and only if $u$ passes the second derivative test for $J$ among compact perturbations that preserve the   relative volume of $\{y\in \mathbb{R}^n:u(y)>0\}$. 

Notably, domains $\Omega$ that are solutions of \eqref{Serrin} and  are weakly stable arise in the context of minimal surface theory. To explain this, let $S\subset \mathbb{R}^{n+1}$ be a closed hypersurface with positive mean curvature bounding an open and bounded set $M\subset \mathbb{R}^{n+1}$. Let $\Sigma \subset \bar M$ be a relative boundary. In particular, $\partial \Sigma \subset S$ bounds a closed subset $S(\Sigma)\subset S$ called the wetting surface of $\Sigma$. $\Sigma$ is called a minimal capillary surface supported on $S$ if the mean curvature of $\Sigma$ vanishes and if $\Sigma$ and $S$ intersect along $\partial \Sigma$ at a constant angle  $\theta\in(0,\pi)$. Minimal capillary surfaces are critical for area given the area of their wetting surface. They are called weakly stable if they pass the second derivative test for area among perturbations that fix the area of their wetting surface. In the case where $n=2$, the  authors \cite{eichmair2026tangential} have shown weakly stable solutions $\Omega$  of \eqref{Serrin} arise as certain blowup limits of weakly stable minimal capillary surfaces supported on $S$ with $\theta$ close to either $0$ or $\pi$. This has led to a characterization of all weakly stable minimal capillary surfaces supported on $S$ with angle close to $0$ or $\pi$; see \cite[Theorem 6]{eichmair2026tangential}.  A fundamental ingredient in the proof is the following rigidity result for weakly stable solutions of \eqref{Serrin}.
\begin{thm}{\cite[Proposition 31]{eichmair2026tangential}} Let $n=2$. Assume that $\Omega$ is a weakly stable solution of \eqref{Serrin} and that $\partial \Omega$ has bounded curvature. Then $\Omega$ is a ball of radius $1$. \label{main theorem 0}
\end{thm}
In this paper, we show that Theorem \ref{main theorem 0} holds in all dimensions, even without the assumption that the curvature of $\partial \Omega$ is bounded. In the case where $f=n$, this gives an affirmative answer to the conjecture of H.~Berestycki, L.~Caffarelli, and L.~Nirenberg under the, from a variational point of view, most natural additional assumption.
\begin{thm} \label{main theorem} 
Let $n\geq 2$.	Assume that $\Omega$ is a weakly stable solution of \eqref{Serrin}. Then $\Omega$ is a ball of radius $1$.
\end{thm}
It is remarkable that Theorem \ref{main theorem} holds in all dimensions. Indeed, J.~Serrin's overdetermined problem \eqref{Serrin} is closely related to the one-phase Bernoulli problem; see \cite{altcafarelli} and also \cite{ChodoshEdelenLi} for a connection between the one-phase Bernoulli problem and minimal capillary surfaces. The one-phase Bernoulli problem is \eqref{Serrin generalized} with $f=0$, i.e., 
\begin{align} \label{Bernoulli} 
	\begin{dcases}
	\quad	-\Delta u=0\qquad& \text{in $\Omega$},\\
	\quad	u>0&\text{in $\Omega$},\\ 
	\quad	u=0 &\text{on $\partial \Omega$, and}\\
	\quad	|Du|=1&\text{on $\partial \Omega$}.
	\end{dcases}
\end{align} 
In the case where $n\leq 3$, stable solutions $\Omega$ of \eqref{Bernoulli} are known to be half-spaces; see \cite{ChanFernandezRealFigalliSerra2026} and, e.g., \cite{KamburovWang}. By contrast, in the case where $n\geq 7$, there are infinitely many geometrically distinct stable solutions of \eqref{Bernoulli}; see \cite{Silva}. This shows that the assumption that $f=n$ in Theorem \ref{main theorem} cannot be dropped altogether. We note that, in a similar context, M.~do Carmo \cite{doCarmo} has conjectured that  constant mean curvature hypersurfaces in $\mathbb{R}^{n+1}$ that pass the second derivative test for area among all hypersurfaces enclosing the same relative volume are round spheres for all $n\geq 2$. By contrast, stable minimal hypersurfaces in $\mathbb{R}^{n+1}$ are known to be hyperplanes in the case where $n\leq 5$ while there are infinitely many geometrically distinct stable minimal hypersurfaces in $\mathbb{R}^{n+1}$ in the case where $n\geq 7$; see, e.g., \cite{doCarmoPeng,Pogorelov,Fischer-ColbrieSchoen,Bernstein4,Bernstein5,Bernstein6} and \cite{Simons}. 
\subsection*{Outline of our arguments} Our approach toward Theorem \ref{main theorem} differs substantially from our proof of Theorem \ref{main theorem 0} in \cite{eichmair2026tangential}, which is based on the Gauss-Bonnet theorem and therefore specific to the case where $n=2$. It also differs substantially from the arguments employed in the recent work of H.~Chan, X.~Fern\'andez-Real, A.~Figalli, and J.~Serra \cite{ChanFernandezRealFigalliSerra2026} on \eqref{Bernoulli}, which are specific to the case where $n=3$ and, in view of the counterexamples found in \cite{Silva}, cannot possibly be adapted to  the case where $n\geq 7$.

First, we suppose, for a contradiction, that $\Omega$ has infinite volume. Given $\lambda>0$, let $A(\lambda)=|B_\lambda(0)\cap \partial \Omega|$ and $V(\lambda)=|B_\lambda(0)\cap  \Omega|$. Using  \eqref{Serrin}, we show that \begin{align} \label{A vs V} A(\lambda)\approx n\,V(\lambda)\end{align}  for suitable numbers $\lambda\gg1$; see Proposition \ref{asymptotics}. In conjunction with a balancing argument, \eqref{A vs V} implies that  $\Omega$ is  stable; see Lemma \ref{weakly stable, stable}. The proof relies on certain gradient estimates for $u$ that we establish in Section \ref{gradient estimates}. We then choose a sequence $\{\eta_k\}_{k=1}^\infty$ of suitable cutoff functions $\eta_k\in C^\infty_c(\mathbb{R}^n)$ with $\eta_k\to 1$ locally smoothly. Applying the stability of $\Omega$ with $\phi=\eta_k$ and using \eqref{A vs V}, we obtain that
$$
\int_{\partial \Omega} \eta_k^2\,(n-1-H)\leq o(1)\,A(\lambda_k)-A(\lambda_k).
$$
Applying the stability of $\Omega$ with $\phi=\eta_k\,\varphi$ where $\varphi=u+|Du|^2/2$ and using \eqref{A vs V} again, we conclude that 
$$
\int_{\partial \Omega} \eta_k^2\,(n-H)\leq -(2+o(1))\,A(\lambda_k);
$$
see \eqref{contradict inequality}. A crucial ingredient here is that $\varphi$ is subharmonic in $\Omega$ and satisfies $\mu(\Omega)\cdot D\varphi=n-1-H$ on $\partial \Omega$; see Lemma \ref{phi}. By contrast, the gradient estimates in Section \ref{gradient estimates} show that $n\geq H$; see Lemma \ref{mean curvature trivial bound}. This leads to a contradiction.

It follows that $\Omega$ has finite volume. In view of the gradient estimates  in Section \ref{gradient estimates}, we show how to adapt the argument of H.~Weinberger \cite{Weinberger} to this setting; see Proposition \ref{infinite volume}. In particular, every component of $\Omega$ is a ball of radius 1. Using the weak stability of $\Omega$, we see that $\Omega$ is connected.

\subsection*{On the use of artificial intelligence (A.I.)} None of the arguments employed here have been suggested by or discussed with A.I. tools. The paper
does not contain text written by A.I. tools. The authors have used Claude Opus 5 to check the typesetting of this paper.
\subsection*{Acknowledgments}  This research was  funded in whole or in part by the
Austrian Science Fund (FWF) [\href{https://www.fwf.ac.at/en/research-radar/10.55776/PAT1307525}{10.55776/PAT1307525}, \href{https://www.fwf.ac.at/en/research-radar/10.55776/PAT2423724}{10.55776/PAT2423724}].

\section{Gradient estimates} 
\label{gradient estimates}
Throughout, we assume that $n\geq 2$. We use $O(1)$ to denote  constants that only depends on $n$.

The arguments in this section are well-known to experts in the field.
\begin{lem} \label{Liouville}
	There holds $\Omega\neq \mathbb{R}^n$. 
\end{lem}
\begin{proof}
	Suppose, for a contradiction, that $\Omega=\mathbb{R}^n$. Then $v\in C^\infty(\mathbb{R}^n)$ given by  
	$$
v(x)= u(x)+\frac12\,|x|^2
	$$
	is  positive and harmonic. By the Liouville theorem, $v$ is a positive constant. Since  $u>0$, this leads to a contradiction. 
\end{proof}
Given $x\in \mathbb{R}^n$ and $r>0$, let $$B_r(x)=\{y\in \mathbb{R}^n:|y-x|<r\}.$$ Given $x\in \Omega$, let  $$r(x)=\operatorname{dist}(x,\partial \Omega).$$ By Lemma \ref{Liouville}, $r(x)<\infty$.  Note that  $B_{r(x)}(x)\subset \Omega$ and $\partial \Omega\cap \partial B_{r(x)}(x)\neq \emptyset$. Moreover, 
$$
\mu(\Omega)(z)=\frac{z-x}{|z-x|}
$$
for all $z\in \partial \Omega\cap \partial B_{r(x)}(x)$.

\begin{lem}[{cp.~\cite[Lemma 25]{eichmair2026tangential}}] \label{distance}
	Let $x\in \Omega$. There holds 
	$$
r(x)\leq 1.
	$$
\end{lem}
\begin{proof}
	Let  $v\in C^\infty(\mathbb{R}^n)$ be given by 
	$$
	v(y)=\frac12\,r(x)^2-\frac12\,|y-x|^2.
	$$ 
	Note that $u-v$ is harmonic in $B_{r(x)}(x)$ and nonnegative on $\partial B_{r(x)}(x)$. 
	By the maximum principle, $u\geq v$ in $B_{r(x)}(x)$. Also note that  $u(z)=v(z)$ for all $z\in \partial \Omega\cap \partial B_{r(x)}(x)$. At such $z$,
	$$
	0\geq \mu(\Omega)\cdot Du-\mu(\Omega)\cdot Dv= -1+r(x).
	$$ 
This completes the proof of the lemma.
\end{proof}
\begin{lem} \label{subharmonic}
There holds
$$
\frac12\,\Delta |Du|^2\geq n
$$
in $\Omega$. 

\end{lem}
\begin{proof}
By  the Bochner formula, 
	\begin{align} \label{Laplacian} 
\frac12\,	\Delta |Du|^2=Du\cdot D\Delta u+|D^2u|^2=|D^2u|^2.
	\end{align} 
By the inequality of arithmetic and geometric means,
	\begin{align} \label{estimate}  
	|D^2u|^2\geq \frac{(\Delta u)^2}{n}=n.
	\end{align} 
	This completes the proof of the lemma.

\end{proof}
\begin{coro}
	\label{ball characterization} 
	Assume that $\Omega$ is connected and that $$\frac12\,\Delta |Du|^2=n.$$  There is $x\in \mathbb{R}^n$ such that $\Omega=B_1(x)$. 
\end{coro}
\begin{proof}
	By assumption, equality holds in \eqref{estimate}. It follows that  $D^2u=-\operatorname{Id}$. Equivalently, there is  $x\in \mathbb{R}^n$ such that
	$$
	u(y)+\frac12\,|y-x|^2
	$$
	is constant in $\Omega$. Since $u=0$ and $|Du|=1$ on $\partial \Omega$, $\Omega=B_1(x),$ as asserted.
\end{proof}
\begin{lem}[{cp. \cite[Lemma 2.3]{wang2015structure}}] \label{prelim gradient bound}
	There holds
	$$
	\sup\{ |Du(x)|:x\in \Omega\}<\infty.
	$$
\end{lem}
\begin{proof}

		Fix $x\in \Omega$. Let   $v:\mathbb{R}^n\to \mathbb{R}$ be given by 
	$$
	v(y)=\frac12\,r(x)^2-\frac12\,|y-x|^2.
	$$ 
	Note that $u-v$ is harmonic in $B_{r(x)}(x)$ and that $u-v\geq 0$ on $\partial B_{r(x)}(x)$ with equality for all $z\in \partial \Omega\cap \partial B_{r(x)}(x)$. 
	By the strong maximum principle, either $u(y)=v(y)$ for all $y\in B_{r(x)}(x)$ or $u(y)>v(y)$ for all $y\in B_{r(x)}(x)$. Since $Dv(x)=0$, we may assume that $u(y)>v(y)$ for all $y\in B_{r(x)}(x)$.
	Let  $$\kappa=\inf\left\{u(y)-v(y):y\in \partial B_{r(x)/2}(x)\right\}.$$
	By the Harnack inequality, 
	\begin{align} \label{harnack} u(x)-v(x)= O(1)\,\kappa.
	\end{align}
Let   $h:\mathbb{R}^n\setminus\{x\}\to \mathbb{R}$ be  the harmonic function given by
	\begin{align*} 
	h(y)=\frac{1}{\log 2}\,\kappa\,(\log r(x)-\log|y-x|)
	\end{align*}
	if $n=2$ and 
		\begin{align*} 
		h(y)=\frac{r(x)^{n-2}}{2^{n-2}-1}\,\kappa\,(|y-x|^{2-n}-r(x)^{2-n})
	\end{align*} 
	if $n\geq 3$. 
	Note that    $$h(y)=\begin{dcases}
	\kappa\qquad&\text{if $|y-x|={r(x)}/2$ and}\\
0	 &\text{if $|y-x|=r(x)$.} 
	\end{dcases}
	$$ By the maximum principle,  $u-v \geq h$ in $\bar B_{r(x)}(x)\setminus  B_{r(x)/2}(x)$ with equality for all $z\in \partial \Omega\cap \partial B_{r(x)}(x)$. At such $z$, there holds 
	$$
	0\geq  \mu(\Omega)\cdot  D(u-v) -\mu(\Omega)\cdot D h=-1+r(x)+\frac{1}{\log 2}\,\frac{\kappa}{r(x)}\quad\,\,\,\,\,
	$$
if $n=2$	and 
	$$
	0\geq  \mu(\Omega)\cdot  D(u-v) -\mu(\Omega)\cdot D h=-1+r(x)+\frac{n-2}{2^{n-2}-1}\,\frac{\kappa}{r(x)}
	$$
	if $n\geq 3$.
It follows that $\kappa=O(1)\,r(x)$. 
	In conjunction with \eqref{harnack} and the gradient estimate for positive harmonic functions,   $$|D (u-v)(x)|=O(1)\, \frac{u(x)-v(x)}{r(x)}=O(1) .$$  
	Since $D v(x)=0$, the assertion follows.
\end{proof}

\begin{coro}[{cf. \cite[Proposition 2.1]{wang2015structure}}] \label{gradient}
	There holds $|Du| < 1$ in $ \Omega$.
\end{coro}
\begin{proof}
Let $0<\varepsilon<1$. By Lemma \ref{subharmonic}, $\psi_\varepsilon\in C^\infty(\bar\Omega)$ given by
$$
\psi _\varepsilon(y)=|Du(y)|^2-\varepsilon\,|y|^2
$$ 	
is strictly subharmonic. By the maximum principle and Lemma \ref{prelim gradient bound},  $\psi _\varepsilon$  attains its maximum on $\partial \Omega$. In particular, $|Du(y)|^2 - \varepsilon\, |y|^2 < 1$ for all $y \in \Omega$. 
 Letting $\varepsilon\to 0$, we obtain that $|Du(y)|\leq 1$ for all $y\in \Omega$. By the strong maximum principle, either $|Du|=1$ in $\Omega$ or $|Du|<1$ in $\Omega$. In view of Lemma \ref{subharmonic}, the former alternative is impossible. This completes the proof of the corollary. 
\end{proof}
\begin{rema}
The estimate
$$
\frac12\,|D u|^2\leq F(u)
$$ for solutions of \eqref{Serrin generalized} where $F\in C^\infty(\mathbb{R})$ is such that $F(0)=1/2$ and $F'=f$ proven in \cite[Proposition 6.3]{wang2015structure} gives
$$
|Du|^2\leq 1+\frac{n}{2}\,u
$$
in the case where $f=n$; see also \cite{Modica}. Note that, unlike  Corollary \ref{gradient}, this estimate is not optimal.
\end{rema}
\begin{coro} \label{function} There holds $0<u<1$ in $ \Omega$.
\end{coro}
\begin{proof}
	This follows from Lemma \ref{distance} and Corollary \ref{gradient}.
\end{proof}

\begin{lem} \label{second derivative identity}
	There holds 
	$$
	D^2u(Du,Du)= H-n
	$$
	on $\partial \Omega$.
\end{lem}
\begin{proof}
There holds
$$
\Delta v=\Delta^{\partial \Omega}v +D^2v(\mu(\Omega),\mu(\Omega))+H\,\mu(\Omega)\cdot Dv
$$
for every $v\in C^\infty(\bar \Omega)$. The assertion follows from this, taking $v=u$.
\end{proof}
\begin{lem}[{cf. \cite[Lemma 6.3]{wang2015structure} and \cite{CaffarelliaKenig}}] \label{mean curvature trivial bound}
	There holds $H\leq n$ on $ \partial \Omega$. 
\end{lem}
\begin{proof}
By Corollary \ref{gradient},
$$
 \mu(\Omega)\cdot D|Du|^2\geq 0
$$	
and $\mu(\Omega)\cdot D|Du|^2=-2\,D^2u(Du,Du)$ on $\partial \Omega$. 
The assertion follows from this and   Lemma \ref{second derivative identity}. 
\end{proof}
\section{Asymptotics}

Given $\lambda>0$, let 
$$
A(\lambda)=|B_\lambda(0)\cap \partial \Omega|\qquad\text{and}\qquad V(\lambda)=|B_\lambda(0)\cap \Omega|.
$$
Note that $V(\lambda)\leq |B_\lambda(0)|= \lambda^n\,|B_1(0)|$ for all $\lambda>0$ and that $V$ is nondecreasing and Lipschitz. For almost every $\lambda>0$, $\partial B_\lambda(0)$ and  $\partial \Omega$ intersect transversely. 
$V$ is differentiable at such $\lambda>0$ and there holds
$$
V'(\lambda)=|\partial B_\lambda(0)\cap \Omega|.
$$
Moreover, by the divergence theorem,
$$
A(\lambda)=-\int_{B_\lambda(0)\cap \partial \Omega} \mu(\Omega)\cdot Du=n\,V(\lambda)+\int_{\Omega\cap \partial B_\lambda(0)}\frac{x}{|x|}\cdot Du. 
$$
Let $E:(0,\infty)\to \mathbb{R}$ be given by 
$$
E(\lambda)=\int_{\Omega\cap \partial B_\lambda(0)}\frac{x}{|x|}\cdot Du.
$$
By Corollary \ref{gradient}, $$|E(\lambda)|\leq V'(\lambda)$$ for all $\lambda>0$ as above and  
\begin{align} \label{ODE} 
A(\lambda)=n\,V(\lambda)+E(\lambda).
\end{align} 
\begin{lem} \label{lambda prep}
	There is a sequence $\{m_k\}_{k=1}^\infty$ of positive integers with $m_k\to\infty$ such that 
	$$
	\frac{V(m_k+2)-V(m_k)}{V(m_k)}=o(1).
	$$
\end{lem}
\begin{proof}
	We may assume that $V(1)>0$. 
For every integer $M\geq 1$, there holds
	\begin{align*} 
	\sum_{m=1}^M \frac{V(m+2)-V(m)}{V(m+2)}&\leq \sum_{m=1}^M\int_{V(m)}^{V(m+2)}\frac{1}{t}\,\mathrm{d}t\\&\leq 2\,\int_{V(1)}^{V(M+2)}\frac{1}{t}\,\mathrm{d}t\\&=2\,\log\frac{V(M+2)}{V(1)}\\&=O(1)\,\log M.
	\end{align*} 
It follows that there is a sequence $\{m_k\}_{k=1}^\infty$ of positive integers with $m_k\to\infty$ such that
$$
\frac{V(m_k+2)-V(m_k)}{V(m_k+2)}=o(1).
$$
In particular, 
$$
\frac{V(m_k+2)}{V(m_k)}=1+o(1).
$$
The assertion follows from these estimates. 
\end{proof}
\begin{prop} \label{asymptotics}
	There is a sequence $\{\lambda_k\}_{k=1}^\infty$ of positive numbers with $\lambda_k\to\infty$ such that
	\begin{align} \label{asymptotics zero}
	\frac{V(\lambda_k+1)-V(\lambda_k)}{A(\lambda_k)}=o(1)
	\end{align} 
	and
	\begin{align} \label{asymptotic property}
	\frac{n\,V(\lambda_k)}{A(\lambda_k)}=1+o(1).
	\end{align}
\end{prop}
\begin{proof}
		Let $\{m_k\}_{k=1}^\infty$ be as in Lemma \ref{lambda prep}. Since $V$ is Lipschitz, 
		there is $ \lambda_k\in (m_k,m_k+1)$  such that $\partial B_{\lambda_k}(0)$ and $\partial \Omega$ intersect transversely and  $$
	V'( \lambda_k)\leq V(m_k+1)-V(m_k).
	$$
Using also that $V$ is nondecreasing,
	$$
	V'( \lambda_k)\leq V(m_k+2)-V(m_k)=o(1)\,V(m_k)\leq o(1)\,V(\lambda_k).
	$$
	In conjunction with \eqref{ODE}, $$A(\lambda_k)=(1+o(1))\,n\,V( \lambda_k).$$ Note that 
	$$
	V(\lambda_k+1)-V( \lambda_k)\leq V(m_k+2)-V(m_k)=o(1)\,V(m_k).
	$$
	Using that 
	$$
	n\,V(m_k)\leq n\,V(  \lambda_k)=(1+o(1))\,A( \lambda_k),
	$$
	the assertion follows.
\end{proof}

\section{Finite volume}
Let $\varphi\in C^\infty(\bar\Omega)$ be given by 
\begin{align} \label{varphidef} 
\varphi(x)=u(x)+\frac12\,|Du(x)|^2.
\end{align} 
\begin{lem} \label{phi}
	There holds 
	$$
	-\Delta \varphi \leq 0
	$$
	in $\Omega$ and 
	$$
	\varphi=\frac12\qquad\text{and}\qquad \mu(\Omega)\cdot D\varphi=n-1-H
	$$
	on $\partial \Omega$. 	
\end{lem}
\begin{proof}
	This follows from Lemma \ref{subharmonic} and Lemma \ref{second derivative identity}.
\end{proof}
\begin{coro} \label{ball characterization 2}
	Assume that $\Omega$ is connected  and that $\varphi=1/2$ in $\Omega$. There is $x\in \mathbb{R}^n$ such that $\Omega=B_1(x)$.
\end{coro}
\begin{proof}
	In view of Lemma \ref{subharmonic}, the proof of Lemma \ref{phi} shows that
	$$
	\frac12\,\Delta |Du|^2=n.
	$$
	The assertion now follows from Corollary \ref{ball characterization}.
\end{proof} 
\begin{lem} \label{phi bound} 
	There holds 
	$
	\varphi< 3/2
	$
	in $\Omega$.
\end{lem}
\begin{proof}
 The assertion follows from   Corollary \ref{gradient} and Corollary \ref{function}.
\end{proof}

The proof of the following proposition adapts that of \cite[Theorem 1]{Weinberger} to the case where $\Omega$ is not a-priori bounded.
\begin{prop} \label{infinite volume}
	Assume that $\Omega$ is connected and  has finite volume. There is $x\in \mathbb{R}^n$ such that $\Omega=B_1(x)$.
\end{prop}
\begin{proof}

  Since $\Omega$ has finite volume,  there is a sequence $\{\lambda_k\}_{k=1}^\infty$ of positive numbers such that $\lambda_k\to \infty$ and 
	\begin{align} \label{volume bounds}
		|( B_{\lambda_k+1}(0)\setminus B_{\lambda_k}(0))\cap \Omega|=o(1)\,\frac1{\lambda_k}.
	\end{align}
	 Let $\eta_k \in C_c^\infty(\mathbb{R}^n),$ $k\geq 1,$ be such that
	 \begin{itemize}
	 \item[$\circ$]  $\eta_k(x)=1$ for all $x\in B_{\lambda_k}(0)$,
	 \item[$\circ$] $ \eta_k(x)=0$ for all $x\in \mathbb{R}^n\setminus B_{\lambda_k+1}(0)$,
	 \item[$\circ$] $|D \eta_k|=O(1)$, and  \item[$\circ$] $|D^2  \eta_k|=O(1)$.
	  \end{itemize}
	  Note that
	  $$
	  -\Delta (x\cdot Du)=2\,n.
	  $$
	   By the divergence theorem,
	 \begin{align*} 
&2\,n\,\int_{\Omega}\eta_k\,u-n\,\int_\Omega\eta_k\,x\cdot Du
\\&\qquad=-\int_{\Omega}\eta_k\,u\,\Delta(x\cdot Du)+\int_\Omega\eta_k\,(x\cdot Du)\,\Delta u 
\\&\qquad =\int_{\partial \Omega}\eta_k\, (x\cdot Du)\,\mu(\Omega)\cdot Du+\int_\Omega u\,D(x\cdot Du)\cdot  D\eta_k-\int_\Omega (x\cdot Du)\,Du\cdot  D\eta_k
\\
&\qquad =-\int_{\partial \Omega}\eta_k\, (x\cdot Du)-\int_\Omega u\,(x\cdot Du)\,\Delta  \eta_k-2\,\int_\Omega(x\cdot Du)\,Du\cdot  D\eta_k.	 
 	 \end{align*} 
Moreover, using that $\operatorname{div}x=n$, 
	 $$
	- \int_{\partial \Omega} \eta_k\, (x\cdot Du)=\int_{\partial \Omega}\eta_k\, \mu(\Omega)\cdot x=n\,\int_{\Omega}  \eta_k+\int_{\Omega} x\cdot D \eta_k.
	 $$
	 By \eqref{volume bounds},
	 $$
	 \int_\Omega x\cdot D \eta_k=o(1).
	 $$
	  By Corollary \ref{gradient}, Corollary \ref{function}, and \eqref{volume bounds},
	 $$
	 \int_{\Omega} u\,(x\cdot Du)\,\Delta \eta_k=o(1).
	 $$
	 	  By Corollary \ref{gradient} and \eqref{volume bounds},
	 $$
	 \int_{\Omega} (x\cdot Du)\,Du\cdot  D\eta_k=o(1).
	 $$
	 By the divergence theorem,
	 $$
	 \int_{\Omega} \eta_k\,x\cdot Du=-n\,\int_{\Omega}\eta_k\,u-\int_{\Omega} u\,x\cdot D \eta_k.
	 $$
	 	 By  Corollary \ref{function} and \eqref{volume bounds},
	 $$
	 \int_{\Omega} u\,x\cdot D \eta_k=o(1).
	 $$
	 By the dominated convergence theorem, using Corollary \ref{function}, we conclude that 
	 \begin{align} \label{equality} 
	 (n+2)\,\int_\Omega u=|\Omega|.
	 \end{align}

	 Suppose, for a contradiction, that $\Omega$ is unbounded. Let $\{x_\ell\}_{\ell=1}^\infty$ be a sequence of points in $\Omega$  such that $|x_\ell|\to \infty$.
	  Since $\Omega$ has finite volume, there holds $r(x_\ell)=o(1)$. By Corollary \ref{gradient}, $u(x_\ell)=o(1)$. Using Corollary \ref{gradient} again, we conclude that $\varphi(x_\ell)\leq 1/2+o(1)$. By Lemma \ref{phi} and the maximum principle, 
	  \begin{align} \label{varphi bound} 
	  \varphi \leq \frac12 
	  \end{align} 
	  in $\Omega$. 
	 By the divergence theorem, 
	 $$
	 \int_{\Omega} \eta_k\,|Du|^2= n\,\int_\Omega \eta_k\,u-\int_\Omega u\,Du\cdot D\eta_k.
	 $$
	By the dominated convergence theorem, using Corollary \ref{gradient}, Corollary \ref{function}, and \eqref{volume bounds}, 
	 $$
	 \int_\Omega |Du|^2=n\,\int_\Omega\,u.
	 $$
	 Thus,
	 $$
2\,\int_\Omega \varphi=2\,\int_\Omega u+\int_\Omega |Du|^2=(n+2)\,\int_{\Omega} u.	 
	 $$
In view of \eqref{equality} and \eqref{varphi bound}, we conclude that $\varphi=1/2$. By Corollary \ref{ball characterization 2}, there is $x\in \mathbb{R}^n$ such that $\Omega=B_1(x)$, contradicting that $\Omega$ is unbounded.

It follows that $\Omega$ is bounded. As in the previous paragraph, we see that there  is $x\in \mathbb{R}^n$ such that $\Omega=B_1(x)$. 
 This completes the proof of the proposition.
\end{proof}
\section{Stability and weak stability}
The following lemma has essentially been observed in \cite[Theorem 3.5]{SternbergZumbrun}.
\begin{lem}
	Let $\eta\in C_c^\infty(\bar \Omega)$. There holds
	\begin{align} \label{error bound} 
		&\int_\Omega |D (|D u|\,\eta)|^2+\int_{\partial \Omega} (H-n)\,\,\eta^2
		\leq \int_{\Omega} |D \eta|^2.
	\end{align} 
\end{lem}
\begin{proof}
	By  the divergence theorem, using \eqref{Laplacian} and Lemma \ref{second derivative identity},
	$$
	\frac12\,\int_\Omega D |Du|^2\cdot D \eta^2=-\int_{\Omega} |D^2u|^2\,\eta^2+\int_{\partial \Omega} (n-H)\,\eta^2.
	$$ 
	Thus,  
	\begin{align*} 
		&\int_\Omega |D(|D u|\,\eta)|^2+\int_{ \partial \Omega} (H-n)\,\eta^2
		\\&\qquad =-\int_{\Omega} |D^2 u|^2\,\eta^2+\int_{\Omega}|D|D u||^2\,\eta^2+\int_{ \Omega} |D u|^2\,|D\eta|^2.
	\end{align*} 
	Note  that $|D^2 u|\geq |D |D u||$ almost everywhere. Using also Corollary \ref{gradient}, we obtain \eqref{error bound}.	
\end{proof}
In the case where $n=2$ and $\partial \Omega$ has bounded curvature, the following lemma has been obtained in \cite[Lemma 26]{eichmair2026tangential}.
\begin{lem} \label{weakly stable, stable}
	Assume that $\Omega$ has infinite volume and is weakly stable. Then $\Omega$ is stable.
\end{lem}
\begin{proof}
	Let $\phi \in C_c^\infty(\bar \Omega)$ be such that
	$$
	\int_{\partial \Omega}\phi =1.
	$$
	Let $\lambda>2$ be such that $\operatorname{spt}(\phi)\subset B_{\lambda-2}(0)$. Let $\{\lambda_k\}_{k=1}^\infty$ be as in Proposition \ref{asymptotics}. We may assume that $\lambda_k>\lambda$ for all $k$.  	 We choose $\eta_k\in C_c^\infty(\mathbb{R}^n),$ $k\geq 1$,  such that 
	\begin{itemize}
		\item[$\circ$]  $\eta_k(x)=1$ for all $x\in B_{\lambda_k}(0)\setminus B_{\lambda}(0)$,
		\item[$\circ$] $ \eta_k(x)=0$ for all $x\in B_{\lambda-1}(0)$ and  $x\in \mathbb{R}^n \setminus B_{\lambda_k+1}(0)$, 
		\item[$\circ$] $0\leq \eta_k\leq 1$, and
		\item[$\circ$] $|D \eta_k|=O(1)$.
	\end{itemize}
	By \eqref{asymptotic property}, using that $V(\lambda_k)\to \infty$, there holds $A(\lambda_k)\to \infty$ and 
	$$
	\int_{\partial \Omega} \eta_k\to\infty.
	$$	
In conjunction with \eqref{asymptotics zero},   \begin{align} \label{good cut off}
	\int_{\Omega} |D\eta_k|^2=o(1)\,\int_{\partial \Omega} \eta_k.
	\end{align} 
Note that $\operatorname{spt} \eta_k\cap \operatorname{spt}\phi=\emptyset$.	Let $\varepsilon_k>0$, $k\geq 1$, be such that 
	$$
	\int_{\partial \Omega} \tilde \phi _k=0
	$$
	where $\tilde \phi _k=\phi-\varepsilon_k\,|Du|\,\eta_k$. Note that
	$$
	\varepsilon_k\,\int_{\partial \Omega} \eta_k=O(1).
	$$
	In particular, $\varepsilon_k=o(1)$.
	 By \eqref{good cut off},
	$$
	\varepsilon^2_k\,\int_{\Omega} |D\eta_k|^2=o(1).
	$$
	By the weak stability of $\Omega$,
	$$
	\int_{\partial \Omega} (n-H)\,\tilde \phi_k^2\leq \int_\Omega |D\tilde \phi_k|^2.
	$$
	In conjunction with \eqref{error bound}, we see that 
	$$
\int_{\partial \Omega}(n-H)\,\phi ^2\leq 	\int_{\Omega} |D\phi |^2+\varepsilon^2_k\,\int_{\Omega} |D\eta_k|^2 =\int_{\Omega} |D\phi|^2+o(1).
	$$
This completes the proof of the lemma.
\end{proof}

\section{Proof of Theorem \ref{main theorem}}

\begin{lem}
	Suppose that $\Omega$ is stable. Let $\eta\in C_c^\infty(\bar \Omega)$. There holds 
	\begin{align} \label{first stability} 
		\int_{\partial \Omega} \eta^2\,(n-H)\leq 	9\,\int_\Omega |D\eta|^2+2\,\int_{\partial \Omega} \eta^2\,(n-1-H).
	\end{align} 
\end{lem} 
\begin{proof}
	Recall the definition \eqref{varphidef} of $\varphi \in C^\infty(\bar \Omega)$. Since $\Omega$ is stable and $\varphi=1/2$ on $\partial \Omega$, 
 $$ \frac14\,\int_{\partial \Omega}(n-H)\,\eta^2\leq \int_{\Omega}|D(\eta\,\varphi)|^2.$$
 Note that
 $$
|D(\eta\,\varphi)|^2=\varphi^2\,|D\eta|^2+2\,\eta\,\varphi\,D\eta\cdot D\varphi+\eta^2\,|D\varphi|^2. 
 $$
 Integrating by parts and using Lemma \ref{phi}, we have 
 \begin{align*} 
2\,\int_\Omega \eta\,\varphi\,D\eta\cdot D\varphi+ \int_\Omega\eta^2\,|D\varphi|^2&=-\int_\Omega \eta^2\,\varphi\Delta\varphi+\int_{\partial \Omega}\eta^2\,\varphi\,\mu(\Omega)\cdot D\varphi 
\\&\leq  \frac12\,\int_{\partial \Omega}\eta^2 \,\mu(\Omega)\cdot D\varphi \\
&= \frac12\,\int_{\partial \Omega}\eta^2\,(n-1-H).
 \end{align*} 
Thus,
 $$
\int_{\partial \Omega} \eta^2\,(n-H)\leq  4\,\int_\Omega \varphi^2\,|D\eta|^2+2\,\int_{\partial \Omega} \eta^2\,(n-1-H).
 $$
The assertion follows from   Lemma \ref{phi bound}.
\end{proof} 

\begin{proof}[Proof of Theorem \ref{main theorem}] 
	Suppose, for a contradiction, that $\Omega$ has infinite volume. By Lemma \ref{weakly stable, stable}, $\Omega$ is stable.  	
Let $\{\lambda_k\}_{k=1}^\infty$  be as in Proposition \ref{asymptotics}. We choose $\eta_k\in C_c^\infty(\mathbb{R}^n),$ $k\geq 1$,  such that 
\begin{itemize}
	\item[$\circ$] $\eta_k(x)=1$ for all $x\in B_{\lambda_k}(0)$,
	\item[$\circ$] $\eta_k(x)=0$ for all $x\in \mathbb{R}^n\setminus B_{\lambda_k+1}(0)$, and 
	\item[$\circ$] $|D\eta_k|\leq 2$.
\end{itemize}  
By     the stability of $\Omega$ and \eqref{asymptotics zero},
$$
 \int_{\partial \Omega} \eta_k^2\,(n-H)\leq \int_\Omega |D\eta_k|^2=o(1)\,A(\lambda_k).
$$
It follows that
$$
\int_{\partial \Omega} \eta_k^2\,(n-1-H)\leq o(1)\,A(\lambda_k)-\int_{\partial \Omega} \eta_k^2\leq -(1+o(1))\,A(\lambda_k).
$$
In conjunction with \eqref{first stability}, using \eqref{asymptotics zero} again, 
\begin{align} \label{contradict inequality} 
\int_{\partial \Omega} \eta_k^2\,(n-H)\leq -(2+o(1))\,A(\lambda_k)+36\,(V(\lambda_k+1)-V(\lambda_k))=-(2+o(1))\,A(\lambda_k).
\end{align} 
By Lemma \ref{mean curvature trivial bound}, the left side is nonnegative. 
Since 
$A(\lambda_k)\to\infty 
$ by \eqref{asymptotic property}, this leads to a contradiction.

It follows that $\Omega$ has finite volume. Applying Proposition \ref{infinite volume} to each component of $\Omega$, we see that $\Omega$ is the union of finitely many  balls of radius 1. If $\Omega$ has at least two components, say $\Omega_1$ and $\Omega_2$, then the function $f=\chi_{\Omega_1}-\chi_{\Omega_2}$ satisfies 
$$
\int_{\partial \Omega}f =0,\qquad\int_\Omega |Df|^2=0,\qquad \text{and}\qquad 
\int_{\partial \Omega} (n-H)f^2=2\,n\,|B_1(0)|.
$$
This contradicts the weak stability of $\Omega$. Thus, $\Omega$ is connected. This completes the proof of the theorem.	
\end{proof}
\begin{appendices}
		\section{First and second variation}
		\label{appendix a}
Let $F\in C^\infty(\mathbb{R})$ and $U\subset \mathbb{R}^n$ be nonempty, open, and bounded. Given $u\in C^\infty(\mathbb{R}^n)$, let  
$$
J_U(u)=\frac12\,\int_{U\cap \{y\in \mathbb{R}^n:u(y)>0\}} |Du|^2-\int_{U\cap \{y\in \mathbb{R}^n:u(y)>0\}}F(u).
$$
Let $\Omega=\{y\in \mathbb{R}^n:u(y)>0\}$. We assume that $\Omega$ is nonempty and that  $Du(y)\neq 0$ 	for all $y\in \partial \Omega$. By the regular value theorem, $\Omega$ is a domain with outward normal 
$$
\mu(\Omega)=-\frac{Du}{|Du|}.
$$
  
  Given $\varepsilon>0$, let 
  \begin{align} \label{ut} \{u_t:t\in(-\varepsilon,\varepsilon)\}\end{align}  be a smooth family of functions $u_t\in C^\infty(\mathbb{R}^n)$ such that $u_0=u$. Let $\Omega_t=\{y\in \mathbb{R}^n:u_t(y)>0\}$. We assume that there is $K\Subset U$ such that  $$\{y\in \mathbb{R}^n:u(y)\neq u_t(y)\}\subset K$$ for all $t\in(-\varepsilon,\varepsilon)$.   Shrinking $\varepsilon>0$, if necessary, we may assume that $Du_t(y)\neq 0$ 	for all $t\in(-\varepsilon,\varepsilon)$ and  $y\in \partial \Omega_t$. By the regular value theorem, $\Omega_t$ is a domain. Let   
   $\phi\in C^\infty_c(\mathbb{R}^n)$ be given by  
$$
\phi=\frac{d}{dt}\bigg|_{t=0}u_t.
$$ 
 Let $\gamma_t$ be the  normal speed of the variation $\{\partial\Omega_t:t\in(-\varepsilon,\varepsilon)\}$ of hypersurfaces with respect to the normal pointing out of $\Omega_t$. Note that \begin{align} \label{gamma0} \gamma_0=\frac\phi{|Du|}.
\end{align}
 We assume that 
$$
\frac{d}{dt}\bigg|_{t=0}|U\cap \Omega_t|=\frac{d^2}{dt^2}\bigg|_{t=0}|U\cap\Omega_t|=0.
$$
It follows that
\begin{align} \label{volume constraints} 
\int_{\partial\Omega}\frac\phi{|Du|}=0\qquad\text{and}\qquad \frac{d}{dt}\bigg|_{t=0}\int_{\partial\Omega_t} \gamma_t=0.
\end{align} 
\begin{lem} \label{first variation} 
	There holds 
	\begin{align} \label{first variation formula} 
	\frac{d}{dt}\bigg|_{t=0}J_U(u_t)=-\int_{\Omega}(\Delta u+F'(u))\,\phi-\frac12\,\int_{\partial\Omega}|Du|\,\phi.
	\end{align} 
	
\end{lem}
\begin{proof}
Using \eqref{gamma0} and \eqref{volume constraints}, we have
	$$
	\frac{d}{dt}\bigg|_{t=0} J_U(u_t)=\int_{\Omega}(Du\cdot D\phi-F'(u)\,\phi)+\frac12\,\int_{\partial\Omega}|Du|\,\phi.
	$$
	The assertion now follows from the divergence theorem.
\end{proof}
Assume that, for every choice of $U$, $u$ is critical for $J_U$ among all variations \eqref{ut}. In particular, the right side of \eqref{first variation formula} vanishes for all $\phi \in C^\infty_c(\mathbb{R}^n)$ with 
$$
\int_{\partial \Omega} \frac{\phi}{|Du|}=0.
$$  It follows that there is $\rho> 0$ such that 
\begin{align} \label{PDE} 
	\begin{dcases}
	\quad	-\Delta u=F'(u)\qquad&\text{in $\Omega$ and}\\
		\quad|Du|=\rho&\text{on $\partial\Omega$}.
	\end{dcases}
\end{align} 
\begin{lem}
	Assume that $u\in C^\infty(\mathbb{R}^n)$ satisfies \eqref{PDE}. 	There holds 
	$$
	\frac{d^2}{dt^2}\bigg|_{t=0}J_U(u_t)=\int_{\Omega}(|D\phi|^2-F''(u)\,\phi^2)+\int_{\partial\Omega}\left(H-\frac{F'(0)}\rho\right)\,\phi^2.
	$$
\end{lem}
\begin{proof}
	Given $t\in(-\varepsilon,\varepsilon)$, let 
	$
	\Phi_t: \partial\Omega\to \partial\Omega_t$ 
be the normal parametrization of $\partial \Omega_t$. By \eqref{gamma0}, 
	$$
	\frac{d}{dt}\bigg|_{t=0}\Phi_t=-\frac{\phi}{|Du|}\,\frac{Du}{|Du|}=-\frac{\phi}{\rho^2}\,Du.
	$$
	Note that 
	$$
	\frac{d}{dt}\bigg|_{t=0} |(Du_t)(\Phi_t)|^2=-\frac2{\rho^2}\,D^2u(Du,Du)\,\phi+2\,Du\cdot D\phi.
	$$
	In conjunction with Lemma \ref{first variation}, we obtain
	\begin{align*} 
		\frac{d^2}{dt^2}\bigg|_{t=0}J_U\left (u_t\right)&=-\int_{\Omega}(\Delta \phi+F''(u)\,\phi)\,\phi-\frac1\rho\,\int_{\partial\Omega}\phi\,Du\cdot D\phi
		 +\frac1{\rho^{3}}\,\int_{\partial\Omega}\phi^2\,D^2u(Du,Du).
	\end{align*} 
By the divergence theorem,
	\begin{align*}
		\int_{\Omega}\phi\,\Delta \phi=-	\int_{\Omega}|D\phi|^2-\frac1{\rho}\,\,\int_{\partial\Omega}D\phi\cdot Du\,\phi.
	\end{align*}
	Moreover, on  $\partial\Omega$, using \eqref{PDE},
	$$
	-F'(0)=\Delta u=\Delta^{\partial \Omega}u+D^2u(\mu(\Omega),\mu(\Omega))+H\,\mu(\Omega)\cdot Du=\frac1{\rho^2}\,D^2u(Du,Du)-H\,\rho.
	$$ 
	The assertion follows from these identities.
\end{proof}

\end{appendices}


\end{document}